\documentclass[a4paper,12pt]{article}
\usepackage{amssymb}
\usepackage{amsmath}
\usepackage{amsthm}
\usepackage{xcolor}
\usepackage{hyperref}
\usepackage{tikz}
\usetikzlibrary{calc}
\usetikzlibrary{decorations.pathreplacing}
\hypersetup{
	colorlinks,
	linkcolor={red!60!black},
	citecolor={green!60!black},
	urlcolor={blue!60!black}
}

\usepackage{geometry}
\newtheorem{Theorem} {Theorem} [section]
\newtheorem{Proposition} [Theorem] {Proposition}
\newtheorem{Lemma} [Theorem] {Lemma}
\newtheorem{Remark} [Theorem] {Remark}
\newtheorem{Definition} [Theorem] {Definition}

\newcommand{\PG}{\mathrm{PG}}

\title{Optimally pseudorandom $K_4$-free graphs}
\author{\renewcommand\thefootnote{\alph{footnote}}
Jie Han\footnotemark[1] \and
\renewcommand\thefootnote{\alph{footnote}}
Ferdinand Ihringer\footnotemark[2] \and
\renewcommand\thefootnote{\alph{footnote}}
Hendrik Van Maldeghem\footnotemark[3]}

\date{26 Sep 2026}

\begin{document}
\maketitle

{\renewcommand\thefootnote{\alph{footnote}}

\footnotetext[1]{School of Mathematics and Statistics,
Beijing Institute of Technology, China.
E-mail: {\tt han.jie@bit.edu.cn}}

\footnotetext[2]{Dept.~of Mathematics,
Southern University of Science and Technology, Shenzhen, Guangdong, China.
E-mail: {\tt ihringer@sustech.edu.cn}}

\footnotetext[3]{Department of Mathematics, Computer Science and Statistics,
Ghent University, Krijgslaan 299-S9, B-9000 Ghent, Belgium.
E-mail: {\tt Hendrik.VanMaldeghem@UGent.be}}}

\begin{abstract}
  We show that optimally pseudorandom $K_4$-free graphs of order $n$ and degree $d = \Theta(n^{4/5})$ exist
  by constructing a graph in the split Cayley hexagon, matching the known upper bound.
  This resolves the first open case for $K_k$-free graphs after $k=3$
  for which Alon gave a tight construction in 1994.

  This has a variety of implications for $K_4$-free pseudorandom graphs.
  Furthermore, it implies an explicit lower bound on the Ramsey number of $r(4, t) \geq t^{1.\overline{6} - o(1)}$,
  improving the previous record by Kostochka, Pudlák, and Rödl of $r(4, t) \geq t^{1.6-o(1)}$.
%
\end{abstract}


\section{Introduction}

Pseudorandom graphs are deterministic graphs that mimic the behavior of random graphs.
We refer to the survey by Krivelevich and Sudakov for background and applications \cite{KrivelevichSudakov06}.
One central notion of pseudorandomness is that of $(n, d, \lambda)$-graphs.
A graph is called an \textit{$(n, d, \lambda)$-graph} if it has order $n$, is $d$-regular,
and its second largest eigenvalue of its adjacency matrix
in absolute value is at most $\lambda$. If $\lambda = O(\sqrt{d})$,
then we call a graph \textit{optimally pseudorandom}.

A simple application of the expander-mixing lemma (cf.\ \cite{KrivelevichSudakov06}) shows that any $K_k$-free
optimally pseudorandom $(n, d, \lambda)$-graph satisfies
\[
  d = O\left( n^{1 - \frac{1}{2k-3}} \right).
\]
So far this bound has only been known to be tight for $k=3$ due to a construction by Alon \cite{Alon94}.
For general $k$, the best known construction is due to Bishnoi, Ihringer, and Pepe \cite{BishnoiIhringerPepe20},
giving an example with
\[
 d = \Theta\left( n^{1 - \frac{1}{k-1}} \right).
\]
Here we resolve the problem for $k=4$, the first open case, using a graph constructed on the split Cayley hexagon.

\begin{Theorem}\label{thm:main}
 There exists an optimally pseudorandom $K_4$-free $(n, d, \lambda)$-graph with $d = \Theta\left( n^{4/5} \right)$.
\end{Theorem}

The existence of this graph had been conjectured in several places, for instance, see \cite{FLS13,KSS04,SSV05}.
Due to a result by Mubayi and Verstra\"ete \cite{MubayiVerstraete24}, Theorem \ref{thm:main} implies $r(4, t) = \Omega(t^3 / (\log t)^4)$,
giving a third construction for this lower bound following Mattheus and Verstra\"ete in 2024 \cite{MattheusVerstraete24}, and Brada\v{c} in 2026 \cite{Bradac26}.
Another related problem is the MaxCut problem for $H$-free graphs \cite{AKS05}.
In this area, Theorem \ref{thm:main} answers the third question in \cite[\S6]{GJS23} for $r=4$, namely,
showing that the minimum surplus of a $K_4$-free graph with $m$ edges is $O(m^{7/9})$.

A $d$-regular graph of order $n$ is called \textit{$(p, \beta)$-bijumbled} if, for all subsets of vertices $X, Y$,
\[
 |E(X, Y) - p |X| |Y|| \leq \beta \sqrt{|X||Y|},
\]
where $E(X, Y)$ denotes the number of edges between $X$ and $Y$, and $p = d/n$ denotes the \textit{edge density}.
The expander-mixing lemma implies that the graph in Theorem \ref{thm:main} is $(p, Cp^3n)$-bijumbled.
Thus, Theorem \ref{thm:main} also shows that the condition $\beta = o(p^{3}n)$ that, by Theorem 1.4 in \cite{Morris25},
guarantees the existence of a $K_4$-factor in $(p, \beta)$-bijumbled graphs, is tight.
See the discussion in the introduction there \cite{Morris25}. Also see \cite[\S1.2]{CFZ14}.

An easy application of the expander-mixing lemma shows that an optimally pseudorandom graph with parameters as in Theorem \ref{thm:main}
has independence number at most $n\lambda/d = O( n^{3/5})$, providing an explicit witness
for $r(4, t) \geq t^{5/3-o(1)}$.

\begin{Theorem}
 There exists an explicit construction for $r(4, t) \geq t^{1.\overline{6}-o(1)}$.
\end{Theorem}

This improves slightly upon the previous best
explicit construction due to Kostochka, Pudl\'ak, and R\"odl who
obtained $r(4, t) \geq t^{1.6 - o(1)}$ \cite{KostochkaPudlakRodl10}.
Interestingly, this construction utilized a generalized quadrangle.
See \cite{IhringerMattheus26} for a brief survey together with
the currently best known general bounds.

\paragraph*{AI Declaration}

For $q=2^h$, $h$ odd, the construction was found by the first author using ChatGPT Astra 6.
The authors subsequently developed, refined, and independently verified
all arguments and take full responsibility for the contents of the paper.
The second author used the ChatGPT 5.6 Chat on High to generate the figures
\ref{fig:hex}, \ref{fig:K4}, and \ref{fig:kantor} based on a text prompt,
and Figure \ref{fig:cnt} based on a drawing.

\section{The construction}

A \textit{generalized hexagon} of order $(s, t)$ is a point-line incidence structure
with the following properties:
\begin{enumerate}
 \item Any point lies on exactly $t+1$ lines.
 \item Any line contains exactly $s+1$ points.
 \item The diameter of the incidence graph is $6$.
 \item The girth of the incidence graph is $12$.
\end{enumerate}
Here the \textit{incidence graph} is the bipartite graph on points and lines
with a point and a line adjacent when incident. See \cite{HVM} for further reading.
The generalized hexagon of order $(1,1)$ is an ordinary hexagon (Fig.\ \ref{fig:hex}).
The incidence graph of a projective plane $\PG(2, q)$ is a generalized hexagon of order $(1, q)$
if we regard its $2(q^2+q+1)$ vertices as points and its $(q^2+q+1)(q+1)$ edges as lines.

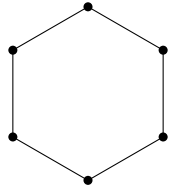
\begin{figure}[ht!]
\begin{center}
\begin{tikzpicture}[scale=.85, every node/.style={font=\small}]
\foreach \i in {0,...,5} {
  \coordinate (v\i) at ({90+60*\i}:1.35);
  \fill (v\i) circle (2pt);
}
\draw (v0)--(v1)--(v2)--(v3)--(v4)--(v5)--cycle;
\end{tikzpicture}
\end{center}

\caption{The generalized hexagon of order $(1,1)$.}

\label{fig:hex}
\end{figure}

As a warm-up, let us count the number of points.
It is easy to count the number of points in a generalized hexagon of order $(s, t)$:
Fix a line $L$. Then $L$ itself has $s+1$ points. Each point on $L$ lies on $t$ more lines,
each of them containing $s$ points. This gives $(s+1)st$ more points.
Now, again, each such point lies on $t$ additional lines,
each of which contains $s$ more points. This gives another $(s+1)s^2t^2$ points. See Fig.\ \ref{fig:cnt}.
Note that the points that we just counted have distance $1$, $3$ and $5$ from $L$
in the incidence graph. Thus, as the girth is $12$, we have not counted any point twice.
The diameter being $6$ implies that we have counted all points. In total, we obtain
that we have
\begin{align}
 (s+1)(1+st+s^2t^2) \label{eq:pts}
\end{align}
points. If we instead fix a point $P$, then we find analogously $(t+1)s$ points
at distance $2$ from $P$ and $(t+1)s^2t$ points at distance $4$ from $P$.
By \eqref{eq:pts}, the number of points at distance $6$ from $P$ is
\begin{align}
 s^3t^2. \label{eq:pts6}
\end{align}

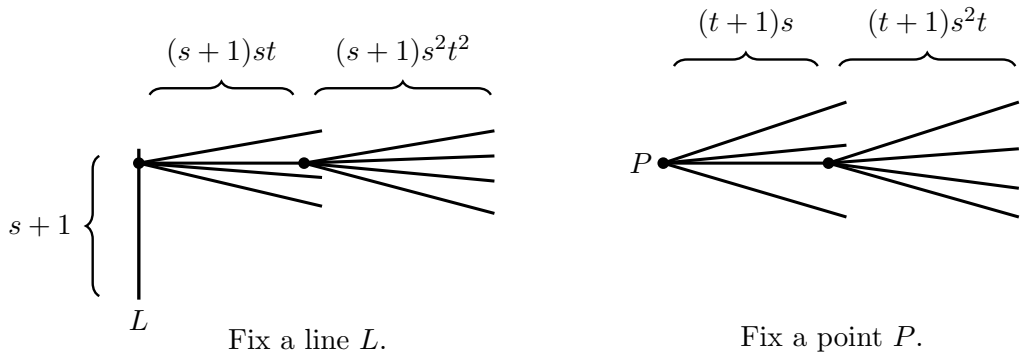
\begin{figure}[ht!]
\begin{center}
\begin{tikzpicture}[scale=.95, every node/.style={font=\small}]

\begin{scope}[xshift=-5.4cm]

\coordinate (A) at (0,0);
\coordinate (B) at (2.3,0);

\fill (A) circle (2.3pt);
\fill (B) circle (2.3pt);

\draw[line width=1.25pt] (0,-1.9)--(0,0.2);
\node[below] at (0,-1.9) {$L$};

\draw[line width=1.15pt] (A)--(2.55,0.45);
\draw[line width=1.15pt] (A)--(B);
\draw[line width=1.15pt] (A)--(2.55,-0.20);
\draw[line width=1.15pt] (A)--(2.55,-0.60);

\draw[line width=1.15pt] (B)--(4.95,0.45);
\draw[line width=1.15pt] (B)--(4.95,0.10);
\draw[line width=1.15pt] (B)--(4.95,-0.25);
\draw[line width=1.15pt] (B)--(4.95,-0.70);

\draw[decorate,decoration={brace,amplitude=5pt},line width=.9pt]
  (0.15,0.95)--(2.15,0.95)
  node[midway,above=6pt] {$(s+1)st$};

\draw[decorate,decoration={brace,amplitude=5pt},line width=.9pt]
  (2.45,0.95)--(4.9,0.95)
  node[midway,above=6pt] {$(s+1)s^2t^2$};

\draw[decorate,decoration={brace,amplitude=5pt},line width=.9pt]
  (-0.55,-1.85)--(-0.55,0.10)
  node[midway,left=6pt] {$s+1$};

\node at (2.35,-2.45) {Fix a line $L$.};

\end{scope}

\begin{scope}[xshift=1.9cm]

\coordinate (P) at (0,0);
\coordinate (Q) at (2.3,0);

\fill (P) circle (2.3pt);
\fill (Q) circle (2.3pt);

\node[left] at (P) {$P$};

\draw[line width=1.15pt] (P)--(2.55,0.85);
\draw[line width=1.15pt] (P)--(2.55,0.25);
\draw[line width=1.15pt] (P)--(Q);
\draw[line width=1.15pt] (P)--(2.55,-0.75);

\draw[line width=1.15pt] (Q)--(4.95,0.85);
\draw[line width=1.15pt] (Q)--(4.95,0.20);
\draw[line width=1.15pt] (Q)--(4.95,-0.35);
\draw[line width=1.15pt] (Q)--(4.95,-0.75);

\draw[decorate,decoration={brace,amplitude=5pt},line width=.9pt]
  (0.15,1.35)--(2.15,1.35)
  node[midway,above=6pt] {$(t+1)s$};

\draw[decorate,decoration={brace,amplitude=5pt},line width=.9pt]
  (2.45,1.35)--(4.9,1.35)
  node[midway,above=6pt] {$(t+1)s^2t$};

\node at (2.35,-2.45) {Fix a point $P$.};

\end{scope}

\end{tikzpicture}
\end{center}

\caption{Counting points in a generalized hexagon.}

 \label{fig:cnt}
\end{figure}

\begin{Lemma}\label{lem:path}
	Let $x,y$ be vertices of the incidence graph with distance at most $5$.
	Then there exists a unique shortest path from $x$ to $y$.
\end{Lemma}
\begin{proof}
 Suppose for a contradiction that there are two shortest paths.
 Then the union of these two paths contains a cycle of length at most $10$,
 contradicting that the girth is $12$.
\end{proof}

\begin{Lemma}\label{lem:nearestln}
  Let $P$ be a point and $M$ a line at distance $5$ from $P$.
  Then there exists a unique line $L$ through $P$ that is closest to $M$.
\end{Lemma}
\begin{proof}
 If there is no line $L$ through $P$ that has distance $4$ from $M$,
 then $M$ has distance at least $6$ from each line through $P$ and, thus,
 at least distance $7$ from $P$, impossible.
 If there is a second line $L'$ through $P$ that has distance $4$ from $M$,
 then there are two shortest paths from $M$ to $P$, contradicting Lemma \ref{lem:path}.
\end{proof}

Let $q$ be a prime power and let $H(q)^D$ denote the \textit{dual split Cayley hexagon},
see \cite[\S3.5.1]{HVM}. This is a generalized hexagon of order $(q,q)$.
Let $\Gamma$ denote the incidence graph of $H(q)^D$.
We provide a brief introduction in Appendix \ref{app:A}.

The distance-$2$ graph $\Gamma_2$ of $\Gamma$, restricted to points, has spectrum
\begin{align}
 q(q+1), \qquad 2q-1, \qquad -1, \qquad -(q+1)\label{eq:spA2}
\end{align}
with multiplicities $1$, $\frac16 q(q+1)^2(q^2+q+1)$, $\frac12 q(q+1)^2(q^2-q+1)$, $\frac13 q(q^4+q^2+1)$, see \cite[Section~6.5]{BCN}.
Similarly, the distance-$4$ graph $\Gamma_4$ has spectrum
\begin{align}
 q^3(q+1), \qquad q(q-2), \qquad -q^2, \qquad q(q+1).\label{eq:spA4}
\end{align}
Next we will define a graph obtained from $\Gamma_4$
by passing to an affine part of it and deleting about half of all the edges.

Fix a point $P_\infty$ and let $X$ denote the points at distance $6$ from $P_\infty$
in $\Gamma$.
Let $L_1, \ldots, L_{q+1}$ denote the lines through $P_\infty$.
Arbitrarily, color $\lfloor (q+1)/2 \rfloor$ of the lines $L_i$ red,
and the remaining $\lceil (q+1)/2 \rceil$ lines $L_i$ blue.
Any line $M_P$ through a point $P$ in $X$,
by Lemma \ref{lem:nearestln}, has a unique line $L_i$ closest to it.
Say that $M_P$ has the color of $L_i$.
Also note that each line through $P$ is closest to a different line $L_i$.
Furthermore, any two points $P$ and $Q$ at distance $4$ are connected
by a unique shortest path, say, $P$-$M_P$-$Z$-$M_Q$-$Q$.

\begin{Definition}
  Define the graph $\Gamma_4^X$ on the vertex set $X$ with two points $P, Q$ adjacent
  if they are adjacent in $\Gamma_4$, and $M_P$ and $M_Q$ have different colors.
\end{Definition}

\begin{Remark}
\begin{enumerate}
 \item
 If one is willing to use the pseudorandomness from random bipartitions
 as, e.g., Conlon in \cite{Conlon17}, then in our case the construction
 can be further simplified
 by simply taking the distance-$4$ graph on points.
 The above choice of $\Gamma_4^X$ lets us control the spectrum quite precisely.
 \item Equivalently, everything can be formulated for lines of $H(q)$.
 We decided to stick to $H(q)^D$ as it seems more natural to think of points
 as vertices.
\end{enumerate}
\end{Remark}

\section{The spectrum}

For any $q$, the graphs are optimally pseudorandom.

\begin{Proposition}\label{prop:spec}
 The graph $\Gamma_4^X$ is an $(n, d, \lambda)$ graph with
 \begin{align*}
  &n = q^5, && d = 2 \lfloor \frac{q+1}{2} \rfloor \lceil \frac{q+1}{2} \rceil (q-1)^2, && \lambda = O(q^2).
 \end{align*}
\end{Proposition}
\begin{proof}
	Equation \eqref{eq:pts6} shows $n=q^5$.
	By a similar argument as for Equation \eqref{eq:pts6}, we can determine $d$.
	Now fix a point $P$ at distance $6$ from $P_\infty$.
	Then $P$ is on $\lfloor \frac{q+1}{2} \rfloor$ red lines,
	each with $q-1$ points at distance $6$ from $P_\infty$.
	Each of these $q-1$ points lies on $\lceil \frac{q+1}{2} \rceil$ blue lines,
	each with $q-1$ points at distance $6$ from $P_\infty$.
	By Lemma \ref{lem:path}, no neighbor of $P$ is counted twice.
	This gives
	\[ \lfloor \frac{q+1}{2} \rfloor \lceil \frac{q+1}{2} \rceil (q-1)^2 \]
	neighbors of $P$.
	Interchanging red and blue gives the remaining neighbors of $P$
	and the count is identical.

	Let $A$ denote the adjacency matrix of $\Gamma_4^X$,
	and let $A_j$ denote the adjacency matrix of the distance-$j$ graph of the point-line-incidence graph restricted to $X$.
	(As $X$ is almost all of the points of $H(q)^D$,
	by Cauchy interlacing, this does not affect the asymptotics of eigenvalue estimates.)
	Define $C$ and $D$ on the vertices $P, Q$ as follows:
	\begin{align*}
	 & C_{PQ} = \begin{cases}
			1 & \text{ if $\mathrm{dist}(P,Q) = 2$, and the unique line through $P,Q$ is red,}\\
			0 & \text{ otherwise.}
	     \end{cases}\\
		&
		D_{PQ} = \begin{cases}
			1 & \text{ if $\mathrm{dist}(P,Q) = 2$, and the unique line through $P,Q$ is blue,}\\
			0 & \text{ otherwise.}
	     \end{cases}
	\end{align*}
	The definition of adjacency in $\Gamma_4^X$ immediately implies that
	\begin{align}
	 A = CD + DC. \label{eq:CD}
	\end{align}
	Thus, it suffices to show that 
	the nonprincipal eigenvalues of $C$ and $D$ are in $O(q)$.
	Each line at distance $5$ from $P_\infty$ closest
	to $L_i$ meets other such lines in a point at distance $4$ from $P_\infty$ (if they meet).
	Thus, for a given $L_i$, these lines are pairwise disjoint in $X$.
	Write $C = \sum C_i$, where $C_i$ is the restriction
	of $C$ to the edges on a red line closest to $L_i$.
	Then, by the above, $C_i$ corresponds to a disjoint union of $K_q$'s,
	that is, it is a direct sum of matrices of the form $J-I$.
	Thus, $C_i+I$ and therefore $C+qI$ are positive semidefinite.
	The same calculation shows that $D+qI$ is positive semidefinite.
	Since $A_2 = C+D$ and the nonprincipal eigenvalues of $A_2$ are in $O(q)$ (see \eqref{eq:spA2}),
	the nonprincipal eigenvalues of $C$ (respectively, $D$) are in $O(q)$.
	Now the assertion on $\lambda$ follows from \eqref{eq:CD}.
\end{proof}

\section{No cliques of size 4}

The graph $\Gamma_4$ can have cliques of size $3$ in two ways.
Either they correspond to three points which are all collinear
with a fourth point, or the three points lie on a (simple) hexagon.
For two points $P, Q$ at distance $4$, we call the unique point $C$
at distance $2$ from $P$ and $Q$ the \textit{center} of $P$ and $Q$ (Fig.\ \ref{fig:K4}).
The picture below illustrates both cases.
Note that, as before, this is a picture of points and lines, not of a graph.

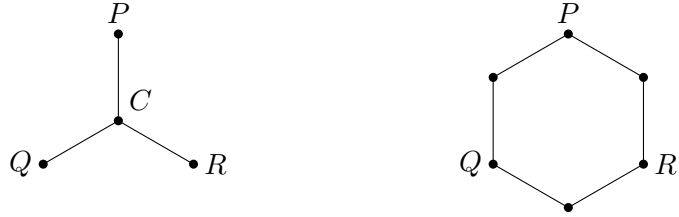
\begin{figure}[ht!]
\begin{center}
\begin{tikzpicture}[scale=.85, every node/.style={font=\small}]
\begin{scope}[xshift=-3.5cm]
\coordinate (o) at (0,0);
\coordinate (a) at (90:1.35);
\coordinate (b) at (210:1.35);
\coordinate (c) at (330:1.35);
\draw (o)--(a) (o)--(b) (o)--(c);
\foreach \n in {o,a,b,c} \fill (\n) circle (2pt);
\node[above right] at (o) {$C$};
\node[above] at (a) {$P$};
\node[left] at (b) {$Q$};
\node[right] at (c) {$R$};
\end{scope}
\begin{scope}[xshift=3.5cm]
\foreach \i in {0,...,5} {
 \coordinate (v\i) at ({90+60*\i}:1.35);
 \fill (v\i) circle (2pt);
}
\draw (v0)--(v1)--(v2)--(v3)--(v4)--(v5)--cycle;
\node[above] at (v0) {$P$};
\node[left] at (v2) {$Q$};
\node[right] at (v4) {$R$};
\end{scope}
\end{tikzpicture}
\end{center}

\caption{Case 1 and Case 2 for a $K_3$ in $\Gamma_4$.}
\label{fig:K4}
\end{figure}

In $\Gamma_4^X$, Case 1 is impossible as not all three lines can have distinct colors.
Thus, we only need to rule out Case 2. This is implied by the following fact
which is stated in Remark 3.7.13 in \cite{HVM}.
We call a set of points $P, Q, R, S$ pairwise at distance $4$
a \textit{Kantor configuration} if all six centers of the six pairs are distinct, that is,
no three have a common center (Fig.\ \ref{fig:kantor}).

\begin{figure}[ht!]
\begin{center}
\begin{tikzpicture}[scale=.95, every node/.style={font=\small}]
\coordinate (P) at (0,1.9);
\coordinate (Q) at (-1.65,-.95);
\coordinate (R) at (1.65,-.95);
\coordinate (S) at (0,0);

\coordinate (CPQ) at (-.825,.475);
\coordinate (CPR) at (.825,.475);
\coordinate (CQR) at (0,-.95);
\coordinate (CPS) at (0,.95);
\coordinate (CQS) at (-.825,-.475);
\coordinate (CRS) at (.825,-.475);

\draw (P)--(CPQ)--(Q)
      (P)--(CPR)--(R)
      (Q)--(CQR)--(R)
      (P)--(CPS)--(S)
      (Q)--(CQS)--(S)
      (R)--(CRS)--(S);

\foreach \n in {P,Q,R,S}
  \fill (\n) circle (2pt);

\foreach \n in {CPQ,CPR,CQR,CPS,CQS,CRS}
  \draw[fill=white] (\n) circle (1.8pt);

\node[above] at (P) {$P$};
\node[left] at (Q) {$Q$};
\node[right] at (R) {$R$};
\node[below] at (S) {$S$};
\end{tikzpicture}
\end{center}

\caption{The Kantor configuration.}

\label{fig:kantor}
\end{figure}
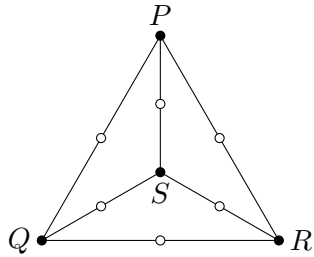

\begin{Proposition}[{\cite[Remark 3.7.13]{HVM}}]\label{prop:noKantor}
	Suppose that $q \equiv 2 \pmod{3}$.
	Then $H(q)^D$ does not possess a Kantor configuration (10 points, 12 lines).
\end{Proposition}

Proposition \ref{prop:noKantor} implies that the graph $\Gamma_4^X$ is $K_4$-free.
Together with Proposition \ref{prop:spec},
this concludes the proof of Theorem \ref{thm:main}.

\section{Conclusion}

The key ingredient for our construction is that the dual split Cayley hexagon $H(q)^D$
does not contain a Kantor configuration when $q \equiv 2 \pmod{3}$,
thus mimicking the construction by Mattheus and Verstra\"ete \cite{MattheusVerstraete24}, where the key observation
was that the \textit{Hermitian unital} does not contain an \textit{O'Nan configuration}.
There is a finite list of nice finite geometries and it would be interesting to
identify more such forbidden configurations in them.
In particular, one natural generalization of the Kantor configuration for
generalized octagons is the \textit{L\"owe configuration}, see \cite{Lowe92}.
It would be interesting to know if its nonexistence implies a similarly interesting result.
Note that here, unlike for generalized hexagons, no generalized octagon is known that does not possess it.

\paragraph*{Acknowledgements}
This research was supported by the National Key R\&D Program of China under grant number 2025YFA1017700.
The first author was partially supported by the National Natural Science Foundation of China (12371341)
and by the Fundamental Research Funds for the Central Universities.
We thank Yuval Filmus and Sam Mattheus for their comments on an earlier version of this manuscript.

\appendix

\section{The split Cayley hexagon} \label{app:A}

\textit{Grassmann coordinates} of lines of the projective space $\PG(d, q)$ of dimension $d$ over the finite field
with $q$ elements are defined as follows. Let $L$ be a line. Pick two points $x,y$ on $L$ with coordinates
$x = (x_0, x_1, \ldots, x_d)$ and $y = (y_0, y_1, \ldots, y_d)$. Then
\[
 p_{ij} = \left| \begin{matrix}
           x_i & x_j \\
           y_i & y_j
          \end{matrix} \right|
\]
is, up to scalar multiples, independent of $x,y$.
In $\PG(6, q)$, the points of the \textit{split Cayley hexagon} $H(q)$ satisfy the equation
\[
 X_0X_4 + X_1X_5 + X_2X_6 = X_3^2,
\]
while the lines are those that are contained in the point set and satisfy the equations
\begin{align*}
 & p_{12} = p_{34}, && p_{20} = p_{35}, && p_{01} = p_{36},\\
 & p_{03} = p_{56}, && p_{13} = p_{64}, && p_{23} = p_{45}.
\end{align*}
The dual split Cayley hexagon is obtained by interchanging points and lines.
See \cite[\S3.5.1]{HVM} for a coordinatization of the split Cayley hexagon
that allows Proposition \ref{prop:noKantor} to be verified directly.
Alternatively, there is a model of the split Cayley hexagon for $q \equiv 2 \pmod{3}$
in $\PG(5, q)$ related to a twisted cubic due to Lunardon \cite{Lunardon}, cf.\ \cite[\S3.7.14]{HVM}.

\end{document}